\documentclass{amsart}

\usepackage{
amsfonts,
amsmath,
latexsym,
amssymb,
enumerate,
verbatim,
mathrsfs,
graphicx,
centernot,
accents, 
color,
}

\usepackage[all]{xy}

\newcommand{\labbel}[1]{\label{#1} [[{\bf #1}]]}  
\renewcommand{\labbel}{\label}

 \definecolor{reedd}{RGB}{0,0,40}

\newcommand{\arxiv}[1]{{\color{reedd}#1}}

\newcommand{\ddd}{\mathrel{\delta}} 
\newcommand{\nd}{\mathrel{\centernot\delta}}

\newcommand{\sharr}{\text{\tiny{$\rightarrow$}}}

\newcommand{\nup}{{\centernot\uparrow}}

\newcommand{\ua}{{\uparrow}}

\newtheorem{theorem}{Theorem}[section]
\newtheorem{lemma}[theorem]{Lemma}

\newtheorem{corollary}[theorem]{Corollary}

\newtheorem*{claim*}{Claim}

\newtheorem*{theorem*}{Theorem}
\newtheorem*{proposition*}{Proposition}
\newtheorem*{corollary*}{Corollary}
\newtheorem*{lemma*}{Lemma}
\newtheorem*{scholion*}{Scholion}

\theoremstyle{definition}
\newtheorem{definition}[theorem]{Definition}

\newtheorem{problems}[theorem]{Problems} 

\theoremstyle{remark}
\newtheorem{remark}[theorem]{Remark}
\newtheorem{remarks}[theorem]{Remarks}
\newtheorem*{remark*}{Remark}
\newtheorem*{remarks*}{Remarks}

\newtheorem{examples}[theorem]{Examples}

\newtheorem*{observation*}{Observation}

 \allowdisplaybreaks[1]

\begin{document}

 \title{Contact semilattices}

\author{Paolo Lipparini} 
\address{Dipartimento di Matematica\\Viale 
dei Contatti di  Ricerca 
\\Universit\`a di Roma ``Tor Vergata'' 
\\I-00133 ROME ITALY}

\email{lipparin@axp.mat.uniroma2.it}

\subjclass{06A12; 06F99, 03G25, 54E05}

\begin{abstract}
We devise  exact conditions under which
a join
semilattice with a weak contact relation
can be semilattice embedded into a
Boolean algebra with an overlap  contact relation,
equivalently, into a
distributive lattice with additive  contact relation.
A similar characterization is proved with respect to 
Boolean algebras and distributive lattices with weak contact,
not necessarily additive, nor overlap.
\end{abstract} 

\keywords{weak contact relation, 
 overlap contact relation, contact join-semilattice,
contact Boolean algebra}

\maketitle

\section{Introduction} \labbel{intro}
 
Contact algebras
are Boolean algebras endowed with a further
contact binary relation; they
 play an important role in 
 region-based theory of space.
Indeed, the theory of contact  algebras
can be seen as a very general 
\emph{first-order theory intended to model topological properties of regions}
\cite[Sect. 4]{BD}.
Contact algebras can also be seen as a point-free algebraization of
proximities, a useful generalization of the notion of topology
\cite{BD,D,N}. 
Contact algebras and generalizations 
 have been used to provide
a logical calculus for compact Hausdorff spaces \cite{BBSV}
and have applications
to the study of nonstandard rules in modal logics \cite{BCGL}. 
See the quoted sources for more references and details.

D{\"u}ntsch, MacCaull, Vakarelov and
 Winter
\cite{DMVW} propose to
drop the notion of complement, thus
considering contact distributive lattices. More radically,
Ivanova \cite{I} provides arguments suggesting the naturalness 
of contact join-\hspace{0 pt}semilattices, namely,
to consider only join, but not complementation, neither meet.
In a nutshell, if $a$ is a region of space,
its complement is possibly not well-defined, since
it is  dependent on  the universe in which
we consider $a$. Moreover, if one considers large but limited regions of space,
both complementation and meet might turn out to be 
inappropriate.

Parallel arguments in favor of the use of 
the join operation only are presented in  \cite{mtt}.
In \cite{mtt} we proposed the project of
detecting which topological properties are preserved
``covariantly''  by image functions
associated to continuous maps. This is more close
to intuition than the present-day customary way of working
contravariantly, say, describing continuity in terms of preservation
of openness through preimages. Indeed, as already stressed by
K. Kuratowski, a function between topological spaces
is continuous if and only if it preserves the adherence relation
between points and subsets. Contact  between subsets
is preserved, as well,
by image functions
associated to continuous maps,
 if we define two subsets $x$ and $y$
to be in contact when $Kx \cap Ky \neq \emptyset  $, where
$K$ denotes closure (this is a basic example of
what in a topological 
context  is called a \emph{proximity} \cite{D}).
We refer to \cite{mtt} for more details. 
As far as the emphasis on join semilattices here is concerned,
notice that if  $f^\sharr$ is the
image function from $\mathcal P(X)$
to $\mathcal P(Y)$ associated to some function
$f: X \to Y$, then $f^\sharr$ preserves unions but not
necessarily
intersections or complements. 

At a very elementary level, the use of 
the join operation alone is also suggested by the fact that 
it is the only operation appearing
in the axioms for an additive contact relation; see Definition \ref{defs} below. 
As another advantage, representation theorems
for semilattices generally  do not need the axiom of choice.
See Remark \ref{ac}, as far as the results presented
here are concerned. 

 A weak definition of contact
in a semilattice guarantees that the semilattice can
be  embedded into a complete lattice, which
can be chosen to have  overlap contact \cite[Theorem 4(b)]{cp}.
On the other hand, the definition  is too weak to guarantee
that a contact semilattice can be  embedded
into a weak contact \emph{distributive} lattice.
Here we find necessary and sufficient conditions under
which the above embeddability can be obtained.
This is also equivalent to being embeddable
into a weak contact Boolean algebra, which can 
be chosen to be atomic and complete. See Theorem \ref{thmb}. 

Moreover, we show in Theorem \ref{thm} 
that embeddability into  an additive contact distributive lattice
 is equivalent to embeddability into a
Boolean algebra with overlap contact.
Thus, as far as weak contact \emph{distributive lattices} are concerned,
\emph{additivity}  alone is enough to guarantee semilattice embeddability
into Boolean algebras with \emph{overlap} contact\footnote{Of course,
 we cannot get lattice embeddability,
since nonoverlap contact relations are maintained by homomorphisms
preserving both contact and meet.}.
Theorem \ref{thm} also provides an equivalent axiomatization
for \emph{Contact join-semilattices} in the sense of \cite{I}.
See Corollary \ref{iv}.
This confirms the usefulness of the notion of a 
Contact join-semilattice. 
On the other hand, Theorem \ref{thmb} provides a larger class
of weak contact semilattices, a class 
which seems to be of interest, as well.
The two classes are distinct, as we will show in Example \ref{ex}(c). 

While, as  summarized above, the main emphasis in this note is about 
very ``regular'' contact semilattices enjoying particularly good properties
and satisfying refined representation theorems,
it is conceivable that also the more general notion of a weak contact semilattice
is interesting and will find significant applications in the future, as already 
hinted in
the final section of \cite{DW}.
There are easily constructed examples which lie beyond
the above-described classes: the   nondistributive
five-element  modular lattice $\mathbf M_3$ with  overlap contact is not additive.
Even if we give  $\mathbf M_3$ an additive contact
structure, certain representation theorems fail.
See Examples \ref{ex}(a)(b). 
Such examples are interesting since  there are many significant
examples of nondistributive lattices having
various kinds of applications 
in many sciences \cite{PP,R}. 

In an even more general situation, weak contact relations on posets, under different
terminology and with entirely different motivations,
appear also as  \emph{event structures}
in computer science, e.~g.,  \cite[Section 8]{WN}.
See \cite[Remark 7]{cp} for explicit details. 

In conclusion, weak contact posets and semilattices
have intrinsic interest. Applications to logics are presented in \cite{I}.
There are also plenty of logical and topological applications 
of contact lattices and contact (Boolean) algebras 
\cite{BD,BBSV,BCGL,DV,DMVW,PH}, among many others, while
logical applications of semilattices, possibly
with further structure, are recalled in \cite[Chapters 6--8]{CHK}. 
This suggests that similar applications of
weak contact join-semilattices will be found, in particular,
with regards to 
 fragments of logics with neither
 negation, nor 
conjunction.
 As a small  logical application,
we use our representation theorems in order to characterize
 the set of universal consequences of the
theory of Boolean algebras with a contact relation  
in the language of contact semilattices. See 
Corollary \ref{conuv}.

\section{Preliminaries and basic definitions} \labbel{prel} 

\begin{definition} \labbel{defs}    
In the present note \emph{semilattices} are always  intended in the sense
of join semilattices with a minimum element $0$.
The semilattice operation will be denoted by $+$. In any
join semilattice the operation $+$ induces a partial order 
$\leq$ defined by $a \leq b$ if $a+b=b$.   
When we speak of a partial order in a semilattice,
we will always mean the order defined above.

The existence of $0$ is assumed only for simplicity.
For example, we shall consider embeddings into
Boolean algebras, which indeed have a $0$.
Were we considering semilattices without $0$,
we should give distinct definitions of overlap,
according to the presence or the absence of $0$. 
See Remark 6 in \cite{cp} for further details. 

A \emph{weak contact relation} (or \emph{basic contact relation})
on some poset $\mathbf S$ 
with $0$ is 
 a binary relation $\delta$  on $S$ such that     
 \begin{align}
\labbel{sym}    \tag{Sym} 
 &  a \ddd b \Leftrightarrow    b \ddd a
\\
\labbel{emp} \tag{Emp}
& a \ddd b  \Rightarrow a > 0  \  \& \ b>0,  
\\
\labbel{ext}    \tag{Ext} 
& a \ddd b \  \& \  a \leq a_1\  \&\  b \leq b_1
\Rightarrow a_1 \ddd b_1,
\\
\labbel{ref}    \tag{Ref} 
& n \neq 0  \Rightarrow  n \ddd n,
  \end{align}
for all $n, a, b, c, a_1, b_1 \in S$. 
We write $a \nd b$ to mean that  $a \ddd b$ does not
hold.  
 The definition of a weak contact relation appears
in \cite{DW}, with main emphasis on lattice-ordered structures.  
Weak contact relations on posets 
 have been studied
in \cite{cp}.

A \emph{weak contact  semilattice}
is a structure   $(S, {+}, 0, { \delta  } )$,
where $(S, {+}, 0)$ is a semilattice with $0$  and $\delta$ 
is a   weak contact relation, as defined above.

The canonical example of a  weak contact  semilattice
is the following.
If $\mathbf S$ is a semilattice 
(or just a poset) with $0$,
then, setting 
\begin{equation}\labbel{ove} \tag{Ove}     
\text{$a \ddd b$ if there is $n \in S$, $n >0$
such that $n \leq a$ and $n \leq b$,} 
   \end{equation}
we get a weak contact relation.
The relation $\delta$ defined in \eqref{ove}  is called the
\emph{overlap} (or \emph{trivial},
or \emph{minimal})  contact relation.

The following property
is frequently required in the definition of a contact
relation (this is the reason for the terminology including
``weak'').
An \emph{additive contact relation} on some semilattice is a weak contact relation
satisfying the following condition. 
\begin{align} 
\labbel{add}    \tag{Add} 
&  a \ddd b+c 
\Rightarrow 
 a \ddd b \text{ or } a \ddd c.
 \end{align} 

The overlap contact relation defined in \eqref{ove}
does not necessarily satisfy 
\eqref{add}. For example, consider 
$\delta$ defined in a $5$-element modular lattice with three
atoms; see Example \ref{ex}(a)  below.
On the other hand, if a semilattice has a distributive lattice order  
and $\delta$ is defined by \eqref{ove},
then \eqref{add} holds, as well. See Lemma \ref{distlem} 
  below.
 An \emph{additive contact
semilattice} is a semilattice with an additive
contact relation.
\end{definition}   

We will consider the following properties of 
a weak  contact semilattice $\mathbf S$.
\begin{align}
 \labbel{d1}    \tag{D1} 
&\begin{aligned} 
& \text{For every }  
a,b ,  c_{0}, c_{1} \in S,
 \text{ if }
 b \leq a +  c_{0}, \ 
 b \leq a +  c_{1} \text{ and }  
 c_{0}\nd c_{1},
\\
& \text{then  } b \leq a.  
\end{aligned} 
\\[8 pt]
 \labbel{d2}    \tag{D2} 
&\begin{aligned} 
& \text{For every } n \in \mathbb N \text{ and } 
a,b ,  c_{1,0}, c_{1,1},  
\dots, c_{n,0},  c_{n,1} \in S,
\\
& \text{if } 
  c_{1,0}\nd c_{1,1},   
\dots,   c_{n,0}\nd c_{n,1} 
\text{ and, for every $f: \{ 1, \dots , n\} \to \{ 0, 1  \}$,}
\\
&\text{either }
b \leq   c_{1,f(1)}  + \dots + c_{n,f(n)}, \text{ or } 
a \leq   c_{1,f(1)}  + \dots + c_{n,f(n)},
\\
& \text{then  } b \nd a.  
\end{aligned} 
  \end{align}  

\begin{remark} \labbel{noti}
 In case $n=0$ in \eqref{d2} we get
 an empty set of summands on the right-hand sides of 
the inequalities. We assume that an empty sum is evaluated to $0$.
Thus the case $n=0$ in \eqref{d2}
is a restatement of 
 \eqref{emp}.   

Moreover, by taking $a=0$ and $b=c_0=c_1 >0$,   \eqref{d1}
implies \eqref{ref}. The case $n=1$ of \eqref{d2}
implies \eqref{ext}. The case $n=1$ of \eqref{d2} implies
also \eqref{sym}, by taking 
$c_{1,1} = b$ and  $c_{1,0} = a$. 
The case $n=2$ in \eqref{d2} with  $c_{1,0} = b$,  $c_{1,1} = a$, 
$c_{2,0} = b$, $c_{2,1} = c$ and $a+c$ in place of $a$  implies
\eqref{add} (in contrapositive form and with $a$ and  $b$ shifted). 
\end{remark}

\begin{lemma} \labbel{lem}
If $\mathbf S$ is a weak contact semilattice 
and $\mathbf S$ satisfies \eqref{d1},
 then $\mathbf S$ satisfies the following condition.
\begin{equation}     
 \labbel{d1+}    \tag{D1+} 
\begin{aligned} 
& \text{For every positive } n \in \mathbb N \text{ and } 
a,b ,  c_{1,0}, c_{1,1},  
\dots, c_{n,0},  c_{n,1} \in S,
\\
& \text{if } 
  c_{1,0}\nd c_{1,1}, \  c_{2,0}\nd c_{2,1}, \  
\dots, \  c_{n,0}\nd c_{n,1} 
\\
&\text{and  }
b \leq 
a +  c_{1,f(1)}  + \dots + 
 c_{n,f(n)}, \text{ for all $f: \{ 1, \dots , n\} \to \{ 0, 1  \}$,}
\\
& \text{then  } b \leq a.  
\end{aligned} 
\end{equation}
 \end{lemma} 

\begin{proof} 
By induction on $n$.
\eqref{d1} is the special case $n=1$ of \eqref{d1+}.

Suppose that we have proved \eqref{d1+} for some specific $n>0$
and suppose that the assumptions of \eqref{d1+} are satisfied for 
$n+1$, say, for certain elements $a,b, \dots, c_{n+1,0}, c_{n+1,1}$. 
 
From $b \leq 
a +  c_{1,f(1)}  + \dots + 
 c_{n,f(n)}+  c_{n+1,f(n+1)}$, for all $f: \{ 1, \dots , n, n+1\} \to \{ 0, 1  \}$,
we get
$b \leq 
a +   c_{n+1,0} + c_{1,g(1)}  + \dots + 
 c_{n,g(n)}$, for all $g: \{ 1, \dots , n\} \to \{ 0, 1  \}$.
By applying \eqref{d1+} in case $n$  with  
$a +   c_{n+1,0} $ in place of $a$, we get
$b \leq a +   c_{n+1,0}$.
Similarly, $b \leq a +   c_{n+1,1}$.
Then \eqref{d1} 
(with   $ c_0 = c_{n+1,0}$ and $ c_1 = c_{n+1,1}$)
gives $b \leq a$.
  \end{proof}

The next lemma appears in \cite[Lemma 2, item 1]{DW}.

\begin{lemma} \labbel{distlem}
In a distributive lattice\arxiv{\footnote{\arxiv{or just in a meet-semilattice
 semidistributive at $0$.}}} 
with overlap contact
the contact relation satisfies \eqref{add}. 
 \end{lemma} 

\begin{proof}
In a  distributive lattice with overlap contact $a \ddd b$
if and only if $ab \neq 0$. Thus if $a \ddd b+c$,
then $ab+ac=a(b+c) \neq 0$, hence either 
$ab \neq 0$ or $ac \neq 0$, thus either $a \ddd b$
or $a \ddd c$.
 \end{proof}

\section{Embedding contact semilattices into 
overlap Boolean algebras} \labbel{emb1} 

\begin{definition} \labbel{emb}    
An \emph{embedding} $\varphi$ 
of (weak) contact semilattices is an injective  
 $0$-preserving map which preserves $+$ and 
such that $a \ddd b$ if and only if $ \varphi (a) \ddd \varphi(b) $,
for all elements $a$ and $b$ in the domain.
 
 In the following theorems we shall deal with ``embeddings''
into models with further structures, e.~g., contact distributive lattices
or contact Boolean algebras. By a slight abuse of terminology,
when we say, for example,  
that  a weak contact semilattice $\mathbf S$ 
can be embedded into 
a contact Boolean algebra $\mathbf  B$, we mean that there is
an embedding from $\mathbf S$ to the  
\emph{reduct}  of $\mathbf  B$ in the language of contact semilattices.   
\arxiv{Notice that, in the above sense, embeddings
are never assumed to preserve meets, or other structure, apart from
 the join operation, the $0$  and the contact relation.} 
\end{definition} 

\arxiv{If $X$ is a topological space, the \emph{contact relation
associated to $X$} on $P(X)$ is defined by $a \ddd b$
if $Ka\cap Kb \neq \emptyset $, for $a, b \subseteq X$. 
More generally, if $\mathbf P$ is a poset with $0$ and with a closure
operation $K$, the associated contact relation is defined by
 $a \ddd b$ if there is $n \in P$, $n >0$ such that both
 $n \leq Ka$ and $n \leq Kb$.}  

\begin{theorem} \labbel{thm}
If $\mathbf S$ is a weak contact semilattice, then the following conditions are equivalent.
 \begin{enumerate}
   \item
$\mathbf S$  can be embedded into 
a  Boolean algebra with overlap contact.
 \item[(1$'$)]
$\mathbf S$  can be embedded into 
 a  Boolean algebra with additive contact.   
\item
$\mathbf S$  can be embedded into 
a  distributive lattice with overlap contact.  
\item[(2$'$)]
$\mathbf S$  can be embedded into 
 an additive contact distributive lattice.  
\item
$\mathbf S$ satisfies \eqref{d1} and \eqref{d2}.
 \item
$\mathbf S$  can be embedded into 
a complete atomic Boolean algebra with overlap contact.  
\arxiv{
\item
$\mathbf S$  can be embedded into 
the contact semilattice associated to some topological space.
\item
$\mathbf S$  can be embedded into 
the contact semilattice associated to some distributive
lattice with additive closure. }   
\end{enumerate}  
\end{theorem}

N.B.: A weak contact on a Boolean algebra 
or on a distributive lattice
is not necessarily additive: see
Example \ref{ex}(c) below.
Hence the additivity assumption is necessary
in clauses (1$'$) and (2$'$). Compare Theorem \ref{thmb} below.

\begin{proof}
(1) $\Rightarrow $  (1$'$) $\Rightarrow $  (2$'$)
and
 (1) $\Rightarrow $  (2) $\Rightarrow $  (2$'$)
are either trivial or   immediate from Lemma \ref{distlem}.

(2$'$) $\Rightarrow $  (3)
Suppose that $\iota : \mathbf S \to \mathbf T$ 
is an embedding given by (2$'$) and 
$\mathbf T$ has the structure of a distributive lattice. 
If $a,b ,  c_{0}, c_{1} \in S$
and   $b \leq a +  c_{0} $,
$  b \leq a +  c_{1}$,
then 
$\iota (b) \leq \iota  (a) + \iota ( c_{0}) $ and 
$\iota (b) \leq \iota  (a) + \iota ( c_{1}) $,
since $\iota$ is a semilattice homomorphism.
Since $\mathbf T$ is a distributive lattice, 
 $\iota (b) \leq (\iota  (a) + \iota ( c_{0}))
(\iota  (a) + \iota ( c_{1}))= \iota  (a) + \iota ( c_{0})\iota ( c_{1}) $.
Since $\iota$ is an embedding, then from 
 $c_{0}\nd c_{1}$ we get 
$ \iota (c_{0}) \nd \iota (c_{1})$,
hence $\iota ( c_{0})\iota ( c_{1}) =0$,
by \eqref{ext} and \eqref{ref}. Thus $\iota (b) \leq \iota  (a)$,
hence $ b \leq a$, since $\iota$ is an embedding. 
This proves \eqref{d1}.

In order to prove \eqref{d2}, assume for simplicity that 
the given embedding is an inclusion, thus we can write, say, $a$
in place of $\iota(a)$.
 We first give a much simpler 
proof of \eqref{d2} under the stronger assumption that  
$\mathbf T$ is a distributive lattice with 
overlap contact, that is, assuming (2).  
Assume the hypotheses of \eqref{d2}
and assume that we are in a distributive lattice with overlap contact.
From \eqref{ref}, \eqref{ext} and   $ c_{1,0} \nd  c_{1,1}$,
 \dots,  $ c_{n,0} \nd  c_{n,1} $
we get $ c_{1,0} c_{1,1}=0 $, 
\dots,  $ c_{n,0}  c_{n,1} = 0$.
Then, by the assumptions and distributivity,
\begin{equation} \labbel{sim}     
ab \leq \prod _{f}  ( c_{1,f(1)}  + \dots + c_{n,f(n)})=
 c_{1,0} c_{1,1} + \dots + c_{n,0}  c_{n,1} =0,
 \end{equation}     
where $f$ varies among all functions from 
 $ \{ 1, \dots , n\} $ to  $ \{ 0, 1  \}$.
  Hence  $a \nd b$, since, by assumption, the relation  is the overlap contact.

 Now we prove \eqref{d2} under the assumption
that $\mathbf T$ is an additive contact  distributive lattice.
Assume the hypotheses of \eqref{d2}
and assume by contradiction that $b \ddd a$.
If, for some  $f: \{ 1, \dots , n\} \to \{ 0, 1  \}$, we have
$b \leq   c_{1,f(1)}  + \dots + c_{n,f(n)}$,
then $b = b c_{1,f(1)}  + \dots + b c_{n,f(n)}$,
by distributivity.
By $b \ddd a$, additivity and symmetry of $\delta$,
we get $b c_{i,f(i)} \ddd a$,
for some $i \leq n$.
Hence the counterexample works if we consider
$b c_{i,f(i)}$ in place of $b$.
Iterating the argument a finite number of times,
it is no loss of generality to assume that, 
for every  $f: \{ 1, \dots , n\} \to \{ 0, 1  \}$,
if $b \leq   c_{1,f(1)}  + \dots + c_{n,f(n)}$,
then $b \leq c_{i,f(i)}$, for some $i \leq n$
which depends on $f$.
Similarly, we can assume the same for $a$. 

Since $b \ddd a$, then $b>0$ and $a>0$, by 
\eqref{emp}. Given some $i \leq n$,
we cannot have both
   $b \leq c_{i,0}$
and
$b \leq c_{i,1}$, by \eqref{ref} and \eqref{ext},
since $c_{i,0} \nd c_{i,1}$. 
We cannot have both
   $b \leq c_{i,0}$
and
$a \leq c_{i,1}$, by \eqref{ext}, since 
$b \ddd a$.
Together with the symmetric arguments,
this shows that, for every $i \leq n$, 
there is at most one between $c_{i,0}$
and $c_{i,1}$ which contains  one 
between $a$ and  $b$.
Choose a function $f: \{ 1, \dots , n\} \to \{ 0, 1  \}$
in such a way that, for every $i \leq n$,
neither $b \leq c_{i,f(i)}$,
nor $a \leq c_{i,f(i)}$.
By   \eqref{d2}, either $b \leq   c_{1,f(1)}  + \dots + c_{n,f(n)}$,
 or $a \leq   c_{1,f(1)}  + \dots + c_{n,f(n)}$.
By the assumptions in the previous paragraph, correspondingly, either
$b \leq c_{i,f(i)}$
or $a \leq c_{i,f(i)}$, for some $i \leq n$,   contradicting
the choice of $f$.

(3) $\Rightarrow $  (1)
Suppose that $\mathbf S = (S, {\leq}, 0,  { \delta  } )$ is a 
weak contact semilattice
satisfying \eqref{d1} and \eqref{d2}.
Consider the  Boolean algebra 
$\mathbf  B = (\mathcal P(S), {\cup}, {\cap}, \emptyset, S, {\complement})$ 
and let  $\varphi: P \to \mathcal P(S)$  
be the semilattice embedding
defined by $\varphi(a) = \nup a = 
\{ \, x \in S  \mid
  a \centernot \leq x\, \} $.
Notice that $\varphi(0) = \emptyset $.  
Let $\mathbf A$ be the quotient  $\mathbf  B/ \mathcal I $,
where $\mathcal I$ is the ideal of 
$\mathbf  B$ generated by the set of all the elements of the form
$\varphi(c) \cap \varphi(d)$, with $c,d \in S$ and $c \nd d$.
If $\pi: \mathbf  B \to \mathbf A$  
  is the quotient homomorphism, then $\kappa= \varphi \circ \pi$
is a  semilattice homomorphism from 
$\mathbf S$ to (the semilattice reduct of) $\mathbf A$.

Endow $\mathbf A$ with the overlap contact relation. 
It is enough to show that $\kappa$ is a contact embedding
from $\mathbf S$ to  $\mathbf A$.
We first need to check that $\kappa$ is injective.
 It is enough to show that
if $ \kappa (b) \leq \kappa  (a)$ in $\mathbf A$,
then $b \leq a$  
 in $\mathbf S$.
If $ \kappa (b) \leq \kappa  (a)$, then  
 $\varphi(b) \subseteq  \varphi (a) \cup i$,
for some $i \in \mathcal I$, that is,
\begin{equation}\labbel{1}
    \varphi(b) \subseteq  \varphi (a) \cup 
(\varphi(c_{1,0}) \cap \varphi(c_{1,1})) \cup \dots \cup 
(\varphi(c_{n,0}) \cap \varphi(c_{n,1})),
   \end{equation}
 for some 
 $n \in \mathbb N$ 
and $c_{1,0}, \dots, c_{n,1} \in S$ such that 
 $c_{1,0} \nd c_{1,1}$, \dots,  $c_{n,0} \nd c_{n,1}$.  
By distributivity, \eqref{1} reads
\begin{equation*}\labbel{2}
\varphi (b) \subseteq \bigcap _{f: \{ 1, \dots , n\} \to \{ 0, 1  \}}
( \varphi (a) \cup \varphi (c_{1,f(1)})  \cup \dots \cup 
\varphi (c_{n,f(n)})),   
   \end{equation*}     
  which holds if and only if 
\begin{equation*}\labbel{3}
\varphi (b) \subseteq 
 \varphi (a) \cup \varphi (c_{1,f(1)})  \cup \dots \cup 
\varphi (c_{n,f(n)}), \text{ for all $f: \{ 1, \dots , n\} \to \{ 0, 1  \}$,}    
   \end{equation*}     
if and only if in $\mathbf S$ 
\begin{equation*}\labbel{4}
b \leq 
a +  c_{1,f(1)}  + \dots + 
 c_{n,f(n)}, \text{ for all $f: \{ 1, \dots , n\} \to \{ 0, 1  \}$,}    
   \end{equation*}     
since $\varphi$  is a semilattice embedding. By Lemma \ref{lem},
$\mathbf S$ satisfies \eqref{d1+}, hence $b \leq a$.    

 Next, we show that $\kappa$ is a $\delta$-embedding.
If $a, b \in S$ and $a \nd b$, then $\kappa(a)  \kappa(b)=0$,
since, by definition, $\varphi(a) \cap \varphi(b) \in \mathcal I$.
Hence $\kappa(a) \nd  \kappa(b)$, since $\delta$ is 
the overlap contact on $\mathbf A$.
For the converse, assume that  $a, b \in S$ and $a \ddd b$,
we need to show that $\kappa(a)  \ddd \kappa(b)$ in $\mathbf A$,
that is, $\kappa(a)  \kappa(b) > 0$, since $\delta$ is 
the overlap contact on $\mathbf A$.
This means $ \varphi (a) \cap  \varphi (b) \notin \mathcal I$.
Assume to the contrary
that $\varphi (a) \cap  \varphi (b) \in \mathcal I$,
that is, 
\begin{equation}\labbel{1b}
   \varphi (a) \cap \varphi(b) \subseteq   
(\varphi(c_{1,0}) \cap \varphi(c_{1,1})) \cup \dots \cup 
(\varphi(c_{n,0}) \cap \varphi(c_{n,1})),
   \end{equation}
 for some $n \in \mathbb N$ 
and $c_{1,0}, \dots, c_{n,1} \in S$ such that 
 $c_{1,0} \nd c_{1,1}$, \dots,  $c_{n,0} \nd c_{n,1}$. 
Notice that, since  $a \ddd b$,
then $a,b >0$, by \eqref{emp}, hence 
  $   \varphi (a) \cap \varphi(b) \neq \emptyset $,
since  $0 \in \varphi (a) \cap \varphi(b)$.
Hence $n \geq 1$ in \eqref{1b}.   
Arguing as in the proof of injectivity of $\kappa$,
the inclusion \eqref{1b}  means
\begin{equation*}\labbel{3b}
\varphi (a) \cap \varphi(b) \subseteq 
 \varphi (c_{1,f(1)}+ \dots + 
c_{n,f(n)}), \text{ for all $f: \{ 1, \dots , n\} \to \{ 0, 1  \}$.}    
   \end{equation*}     
By taking complements, recalling that  $\varphi(a)= \nup a$
and setting  $\ua a = \{ \, x \in S  \mid
 x \geq a \, \} $, we obtain
\begin{equation*}\labbel{4b}
\ua a \cup \ua  b \supseteq 
 \ua  (c_{1,f(1)}+ \dots + 
c_{n,f(n)}), \text{ for all $f: \{ 1, \dots , n\} \to \{ 0, 1  \}$,}    
   \end{equation*}     
that is, 
\begin{align*}
&\text{for every $f: \{ 1, \dots , n\} \to \{ 0, 1  \}$, either}
\\
& c_{1,f(1)}+ \dots + c_{n,f(n)} \geq a, \text{ or } 
 c_{1,f(1)}+ \dots + c_{n,f(n)} \geq b.
 \end{align*} 
   
By \eqref{d2}, this implies $a \nd b$,
contradicting $a \ddd b$.   

(4) $\Rightarrow $  (1) is trivial. As far as
(1) $\Rightarrow $  (4) is concerned, notice that, as well-known, every 
Boolean algebra can be extended to a complete atomic Boolean algebra.
Embed the algebra given by (1) into a  complete atomic Boolean algebra,
and give this larger algebra, too, the overlap contact relation.
Since Boolean embeddings preserve meets and since, with overlap contact
relations, $a \ddd b$ is equivalent to 
 $ab>0$
(in lattices, hence in Boolean algebras),
then the embedding preserves the contact, too.   

\arxiv{ (4) $\Rightarrow $  (5) Since
a complete atomic Boolean algebra $\mathbf  B$ is isomorphic
to a field of sets, say, $\mathcal P(X)$,
if we give $X$ the discrete topology,
the overlap contact on $\mathbf  B$ is the same as the contact induced
by the topology on $X$.

(5) $\Rightarrow $  (6)  $\Rightarrow $  (2$'$) are trivial.
Indeed, in a distributive lattice with an
additive closure operation $K$, the associated
contact relation is additive.}  
\end{proof}

\begin{remarks} \labbel{bou} 
(a) The  proof
of Theorem \ref{thm} (3) $\Rightarrow $   (1),
as given,
does not work for bounded semilattices, i.e., semilattices with also a maximum $1$,
which is supposed to be preserved by homomorphisms.
However, if $\mathbf S$  has a maximum $1$, it is enough to consider 
 as $\mathbf  B$ the Boolean algebra of subsets of $\mathbf S \setminus \{ 1 \} $
in the above proof, in order show that Theorem \ref{thm}
holds for bounded semilattices, as well.  

(b)   For bounded semilattices, the implication (2$'$) $\Rightarrow $   (4)
 (hence the equivalence of (1) - (2$'$) and (4))
 in Theorem \ref{thm}
 can
be obtained also as a consequence of \cite[Proposition  3.8
and Theorem  5.6]{I}. 
\end{remarks}

Ivanova \cite[Definition 3.1]{I} introduced \emph{Contact join semilattices},
which are bounded semilattices with a binary relation
satisfying \eqref{emp}, \eqref{sym}
((C1) and (C4) in the terminology from \cite{I}),
as well as some further conditions 
$(C^1_{m,i})$ and   $(C_{n,i})$,
which we will not report here.
Then in \cite[Theorem 5.6]{I} a representation
theorem is proved, corresponding to 
clause (4) in Theorem \ref{thm} here. 
Hence, by Remarks \ref{noti}  and \ref{bou}(a),  we get the following corollary,
which provides an alternative axiomatization for
Ivanova's Contact join semilattices.
 
\begin{corollary} \labbel{iv}
A bounded semilattice $\mathbf S$ with a further binary relation is a
Contact join semilattices in the sense of \cite[Definition 3.1]{I}
if and only if $\mathbf S$ satisfies \eqref{d1} and \eqref{d2}. 
 \end{corollary}

\section{Embedding contact semilattices into 
(nonoverlap) Boolean algebras} \labbel{emb2}

\arxiv{Condition  \eqref{d2} is not needed} 
in order to get that
a weak contact semilattice can be embedded into
a weak contact distributive lattice (possibly, with non-overlap 
weak contact relation).
Recall the conventions on embeddings
from Definition \ref{emb}. 

\begin{theorem} \labbel{thmb}
If $\mathbf S$ is a weak contact semilattice, then the following conditions are equivalent.
 \begin{enumerate}
   \item
$\mathbf S$  can be embedded into 
a weak contact Boolean algebra.
\item
$\mathbf S$  can be embedded into 
a weak contact distributive lattice.
\item
\arxiv{$\mathbf S$ satisfies \eqref{d1}.}
  \item
$\mathbf S$  can be embedded into 
a weak contact complete atomic Boolean algebra.  
 \end{enumerate}  
\end{theorem} 

\begin{proof}
(1) $\Rightarrow $  (2) and 
(4) $\Rightarrow $  (1) 
are trivial.

 The proof that $\mathbf S$ satisfies  \eqref{d1} 
in Theorem \ref{thm} (2$'$) $\Rightarrow $  (3) does not use
additivity and does not use
the assumption that $\delta$ is the overlap relation, hence the corresponding
implication holds in the present case, as well. 
\arxiv{This shows (2) $\Rightarrow $  (3).} 

(3) $\Rightarrow $  (1) Define the Boolean algebras $\mathbf  B$, $\mathbf A$ 
and the homomorphisms $\varphi$, $\pi$ and $\kappa$ 
as in the corresponding case in the proof of Theorem \ref{thm}.
Since the proof there that $\kappa$ is injective uses only 
\eqref{d1}, we get that $\kappa$ is injective in the present case, as well. 
 The definition of the weak contact in $\mathbf A$ needs to be modified
in the present situation. For $x,y \in A$, let  $ x \ddd  _{\mathbf A} y$ 
if either $xy > 0$, or there are $a,b  \in S$ such that $a \ddd b$
and $ \kappa (a) \leq x$, $\kappa(b) \leq y$.    

The properties \eqref{sym}, \eqref{ext} and \eqref{ref}
for  $\ddd  _{\mathbf A}$ are immediate. 
We have already proved that $\kappa$ 
is an embedding, hence if $a \neq 0$,
then $\kappa(a) \neq 0$, thus \eqref{emp} holds in
$\mathbf A$, since it holds in $\mathbf S$.

It remains to prove that $\kappa$ is a $\delta$-embedding.
If $c \ddd d$, then 
$ \kappa (c) \ddd  _{\mathbf A} \kappa(d) $
by 
the definition of $\ddd  _{\mathbf A}$.
On the other hand, suppose $c \nd d$ and,
by contradiction,  $ \kappa (c) \ddd  _{\mathbf A} \kappa(d) $.
Since $c \nd d$, then $\kappa (c) \kappa (d)=0$,
by the definition of $\mathcal I$ in the proof of
Theorem \ref{thm} (3) $\Rightarrow $  (1).   
Hence, by the definition of $\ddd  _{\mathbf A}$,
 there are $a,b \in S$ such that $a \ddd b$
and $ \kappa (a) \leq \kappa (c)$, $\kappa(b) \leq \kappa (d)$.
This means $\varphi (a) \subseteq \varphi (c) \cup 
(\varphi(c_{1,0}) \cap \varphi(c_{1,1})) \cup \dots \cup 
(\varphi(c_{n,0}) \cap \varphi(c_{n,1}))$, for elements
$c_{1,0}, \dots , c_{n,1} \in S$
such that 
$c_{1,0} \nd c_{1,1}$, \dots,  $c_{n,0} \nd c_{n,1}$,
and 
$\varphi (b) \subseteq \varphi (d) \cup 
(\varphi(d_{1,0}) \cap \varphi(d_{1,1})) \cup \dots \cup 
(\varphi(d_{i,0}) \cap \varphi(d_{i,1}))$, for elements
$d_{1,0}, \dots $ satisfying the corresponding properties.
As custom by now,  we get
$\varphi (a)  \subseteq \varphi (c) \cup 
 \varphi (c_{1,f(1)}) \cup  \dots \cup 
\varphi (c_{n,f(n)}) $. 
Since $\varphi$  is a semilattice embedding, 
then 
$a \leq c +   c_{1,f(1)} +   \dots + c_{n,f(n)} $
in $\mathbf S$, for every $f: \{ 1, \dots , n\} \to \{ 0, 1  \}$.  
\arxiv{By \eqref{d1+}, $a \leq c $ and similarly  
$b \leq d$, thus $a \nd b$ by \eqref{ext}, a contradiction. }

(1) $\Rightarrow $  (4)
Suppose that $\mathbf S$ can be embedded in 
a weak contact Boolean algebra $\mathbf A$.
Embed (the Boolean reduct of) $\mathbf A$ into 
some atomic complete Boolean algebra $\mathbf  C$
by, say, a  Boolean embedding $\chi$. 
Let $u \ddd v$ in $\mathbf  C$ 
if either $uv>0$, or $u\geq \chi(a)$ and   $v\geq \chi(b)$,
for some $a, b \in A$ such that $a \ddd b$.
If $c,d \in A$ and $c \ddd d$, then 
$\chi(c) \ddd \chi(d)$ in $\mathbf  C$, by definition.
 If $c,d \in A$ and $c \nd d$, then 
$cd=0$, by \eqref{ext} and \eqref{ref},
thus   $\chi(c)  \chi(d) =0$, since $\chi$
is a Boolean homomorphism. If, by contradiction,
$\chi(c) \ddd  \chi(d)$, then
$\chi(c)\geq \chi(a)$ and   $ \chi(d)\geq \chi(b)$,
for some $a, b \in A$ such that $a \ddd b$.
Since $\chi$  is a Boolean embedding, $c \geq a$ and
$d \geq b$, thus $c \ddd d$, a contradiction.
We have proved that $\chi$ is a $\delta$-embedding.    
It is elementary to see that $ \ddd$ is a weak contact
relation on $\mathbf  C$, hence, by composing the two
embeddings, we get an embedding of $\mathbf S$ into
the atomic and complete weak contact Boolean algebra $\mathbf  C$. 
\end{proof}

\begin{remark} \labbel{ac} 
   The implications (1) $\Rightarrow $  (4)  in
Theorems \ref{thm} and \ref{thmb}  
are the only parts   in the present note
in which we have used 
a consequence of the axiom of choice, namely, \cite[Form 14 B]{HR}.
\arxiv{This remark involves also clause (5) in Theorem \ref{thm}.} 

Formally, a version of choice is used also in the proof of Corollary \ref{iv},
but we expect that  a more direct and choice-free proof can be found. 
\end{remark}

\section{Further remarks and some examples} \labbel{exe} 

 Let $T$ be the 
theory of Boolean algebras with an overlap  contact relation.  
The representation Theorem \ref{thm} 
can be used in order to provide
a characterization of the set of universal consequences of $T$ 
in the language of contact semilattices.
For the statement of the
next theorem, observe that clause \eqref{d1} can be expressed as a first-order
universal sentence and \eqref{d2} can be considered a universal
theory consisting of a countable set of sentences.  

\begin{corollary} \labbel{conuv}
Suppose that  $\varphi$  is a universal sentence
in the language of contact semilattices.
Then the following conditions are equivalent.
 \begin{enumerate}  
 \item 
$\varphi$  
is a logical consequence of the theory of Boolean algebras with
an overlap contact relation. 
\item
$\varphi$  is a logical  consequence of the theory
of distributive lattices with an additive contact relation.
\item
$\varphi$  is a logical  consequence of \eqref{d1} and \eqref{d2}. 
  \end{enumerate}  
Moreover,
the following conditions are equivalent.
 \begin{enumerate}
\setcounter{enumi}{3}
 \item 
$\varphi$  
is a logical  consequence of the theory of Boolean algebras with
a weak contact relation. 
\item
$\varphi$  is a logical  consequence of the theory
of distributive lattices with a weak contact relation.
\item
  \arxiv{$\varphi$  is a logical  consequence of \eqref{d1}.} 
\end{enumerate}
 \end{corollary}

 \begin{proof} 
(1) $\Rightarrow $  (3). 
By Remark \ref{noti} and Theorem \ref{thm} (3) $\Rightarrow $  (1),
if some semilattice $\mathbf S$ with a binary relation satisfies 
\eqref{d1} and \eqref{d2}, then $\mathbf S$ can be embedded
into some  Boolean algebra $\mathbf  B$ with
overlap contact relation. This means that $\mathbf S$ is isomorphic
to some substructure of a reduct of $\mathbf  B$, hence $\mathbf S$ 
satisfies all the universal sentences satisfied by this reduct of $\mathbf  B$.
Thus if $\varphi$  is a universal consequence of the 
theory of Boolean algebras  with
an overlap contact relation, then 
 every model of \eqref{d1} and \eqref{d2}
satisfies $\varphi$, since $\varphi$  is in the language of contact semilattices.
 Then the  
completeness theorem implies that $\varphi$ 
is a consequence of \eqref{d1} and \eqref{d2}.

(3) $\Rightarrow $  (2) By Theorem \ref{thm} (2$'$) $\Rightarrow $  (3), every
  distributive lattices with an additive contact relation satisfies
\eqref{d1} and \eqref{d2}.

(2) $\Rightarrow $  (1) is trivial.

The equivalences of (4) - (6) is proved in a similar way, using
Theorem \ref{thmb}.  
\end{proof}  

By Example \ref{ex}(c) below, the set of formulas $\varphi$ 
for which (1) - (3) hold in the previous theorem
is distinct  from the set of formulas $\varphi$ 
for which (4) - (6) hold.

Corollary \ref{conuv} can be reformulated to deal
only with the language of contact semilattices.
The theory of Boolean algebras can be expressed
in the language of join semilattices, asserting the existence
of meets by means of a $\forall\exists\forall$ sentence, and similarly
for complementation.
Under the above axiomatization, the set of universal consequences of the theory
of Boolean algebras with an overlap (weak contact) relation
is the set of universal consequences of \eqref{d1} and \eqref{d2}
\arxiv{(\eqref{d1} alone). }

\begin{examples} \labbel{ex} 
Let   $\mathbf M_3$ be the $5$-element
modular lattice  with $3$ atoms $a$, $b$ and $c$.

(a) 
If $\mathbf M_3$  is given the  overlap contact relation,
then $\mathbf M_3$  is a weak not additive  contact lattice.
Similarly if we set $a \ddd b$, $c \nd a$, $c \nd b$
and the symmetrical relations.   
Indeed, $c \ddd a+b=1$.

(b) In  $\mathbf M_3$ set 
$a \ddd b$, $a \ddd c$, $b \nd c$ and symmetrically.
Call the resulting model  $\mathbf M_3^ \delta $.
Then $\mathbf M_3^ \delta $ is an additive
contact lattice. However,
$\mathbf M_3^ \delta $ cannot be semilattice embedded into a 
distributive weak contact lattice, since 
$b \nd c$ implies $b c=0$ in any weak contact lattice
and this would give $a=a+bc=(a+b)(a+c)=1$.
Moreover, we have 
$b \leq a+b$,  $b \leq a+c$ and $b \nd c$,
but it is not the case that $b \leq a$.
This shows that an additive contact lattice
does not necessarily satisfy \eqref{d1} (take $c_0=b$, $c_1=c$).  

By \cite[Theorem 4 (b)]{cp} $\mathbf M_3^ \delta $
can be semilattice embedded  
 into a bounded complete weak contact
lattice $\mathbf S$ with overlap  contact relation.
By the above comments, $\mathbf S$ cannot be chosen
to be a distributive lattice. 
 On the other hand, we can take $\mathbf S$ to be a modular lattice:
let $\mathbf S$ be the lattice of subgroups 
of the product $\mathbb Z_4 \times \mathbb Z_4$
of the cyclic group of order $4$. Let $\varphi(b)= \mathbb Z_4 \times \{ 0 \} $,
$\varphi(c)= \{ 0 \} \times \mathbb Z_4 $ and $\varphi(a)$
be the subgroup  of $\mathbb Z_4 \times \mathbb Z_4$
consisting of those pairs with even difference.

(c) Let $\mathbf  B_8$ be the $8$-element Boolean algebra 
with three atoms $a$, $b$ and $c$, with $ c \nd a$,
$c \nd b$, the symmetric relations and all the other pairs of nonzero elements
$\delta$-related.  The weak contact on   $\mathbf  B_8$ is
 not additive, since
$c \ddd a+b $ but neither
$c \ddd a$ nor $c \ddd b$. 
By Theorem \ref{thmb}
$\mathbf  B_8$ satisfies \eqref{d1}. 
\arxiv{Thus  \eqref{d1} 
does not
imply \eqref{add}, even in weak contact Boolean algebras.
 In particular,  \eqref{d1} 
does not
imply \eqref{d2}, since in Remark \ref{noti} we  have showed that \eqref{d2}
implies  \eqref{add}. }

The above example shows that Theorems \ref{thm}
and Theorem \ref{thmb} have distinct ranges of application.

 (d) We now show  that 
\eqref{add} together with \eqref{d1}
do not imply \eqref{d2}.
   
Consider the  join semilattice $\mathbf S$  freely generated by
six elements $c,d, e,f, \allowbreak x,y$ with the relations
\begin{equation} \labbel{fd2}    
\begin{aligned}  
x&\leq c+e,   &  \quad   x&\leq d+f ,      &  
 \quad    x&\leq c+d,     &  \quad    x&\leq e+f,     
\\
  y &\leq c+f, 
   &      y &\leq d+e,
&      y &\leq c+d,
   &      y &\leq e+f. 
 \end{aligned}
 \end{equation}      
The first relation is intended to mean
$x+c+e=c+e$, and similarly for the other relations. 
Elements of $S$ are formal  sums
of subsets of  $\{c,d,e,f,x,y \}$,
including the empty sum $0$,  
with the reducing rules $x+c+e=c+e$ etc.,
modulo associativity, commutativity and idempotence.
Since the reducing rules do not modify 
the set of elements from $\{ c,d,e,f\}$ 
appearing in the expressions, and since the rules
only eliminate either $x$ or $y$, then the final
result of a sequence of reductions is uniquely determined
(in formal terminology, we have a unique normal form). 
The same argument shows that $\mathbf S$ is actually a join 
semilattice, and that if $\mathbf F$ is the
free join semilattice on $\{c,d, e, f  \}$, then
no pair of elements from $F$ are identified by any
chain of reductions. In other words, 
 $\mathbf S$ a
semilattice extending $\mathbf F$.

More explicitly, $\mathbf S$ is the union of the 
free join semilattice on $\{c,d, e, f  \}$,
plus the elements
$x$,  $x+c$, $x+d$, $x+e$,  $x+f$,
$x+d+e$, $x+c+f$, plus
$y$ and the symmetrical sums, 
plus  $x+y+c$,
$x+y+d$, $x+y+e$,   $x+y+f$.

Define an additive contact relation on $S$ 
by setting 
$c \nd d$, $e \nd f$, the symmetrical relations, and letting all the other pairs
of nonzero elements be $\delta$-related.
From the fact that each element is $\delta$-unrelated
with at most one nonzero element,
 it follows that
$\mathbf S$ satisfies \eqref{add}.
On the other hand, \eqref{d2}
is not satisfied in $\mathbf S$: take
$n=2$, $c _{1,0} = c $, 
$c _{1,1} = d $, $c _{2,0} = e $, $c _{2,1} = f $, 
and $x$, $y$ in place of $a$ and $b$ in \eqref{d2}.   
The definition of $\delta$ and the first two columns in \eqref{fd2}
 witness that the assumptions in \eqref{d2} are satisfied.
However, the conclusion fails, since $y \ddd x $.  

The last  two columns in \eqref{fd2}
have been added in order to have \eqref{d1}
satisfied. Indeed, from $x \leq c+ e \leq c+d+ e$
and  $x \leq d+ f  \leq c+d+ f$, we get  $x \leq c+d$,
provided \eqref{d1} is satisfied
(take $b=x$, $a=c+d$, $c_0=e$ and $c_1=f$). 

  In order to check that \eqref{d1} is actually satisfied in
$\mathbf S$ in all the remaining cases, we need to perform some computations.
A general element of $S$ has the form
$w +b$, where  
$ w \in \{ 0, x, y, x+y\}$
and $b \in F$, recalling that 
$\mathbf F$ is the
subsemilattice  of $\mathbf S$ generated by $\{c,d, e, f  \}$. 
If $ w, z \in \{ 0, x, y, x+y\}$,
let $w \setminus z$ be the sum of those
variables which appear in the sum giving $w$,
but not in the sum giving $z$. E.~g.,    
$(x+y) \setminus x = y$; 
 $(x+y) \setminus 0 = x+y$;
$x \setminus y = x$;
$x \setminus (x+y) = 0$.
We now check that if 
$ w, z \in \{ 0, x, y, x+y\}$
and $a, b \in F$, then 
\begin{equation}\labbel{sum} 
w+b \leq z+ a \quad \text{ if and only if  } \quad 
b \leq a \text{ and } w \setminus z \leq a.  
  \end{equation}    
Indeed, 
by convention, 
$w+b \leq z+ a$ means
\begin{equation}\labbel{uf}    
w+b + z+ a = w+z+ b+a = z+ a.
  \end{equation} 
If \eqref{sum} holds, then necessarily $b+a=a$, since,
as commented above, the presence of $x$ or $y$
never deletes the presence of an element from $F$.  
If $b+a=a$, that is $b \leq a$, then \eqref{uf} reads
$ w+z+ a = z+ a$, hence all the variables in 
$w$ but not in $z$ should disappear on the left-hand side, and this happens
exactly when $w \setminus z \leq a$.
Conversely,  if $b \leq a$, then 
$b+a=a$, hence  $z+b+a=z+a$
and if, furthermore,  
$w \setminus z \leq a$, then 
 $w+z+b+a=z+a$, since all the variables in 
$w$ not in $z$ are absorbed by $a$;
the remaining variables in $w$    
are already present in $z$, hence are absorbed by $z$.  

With \eqref{sum} at our disposal, we now can prove
that \eqref{d1} holds in $\mathbf S$.
The only possible choices for   $c_0$ and $c_1$ in \eqref{d1} are 
either $c_0=c$, $c_1=d$, or $c_0=e$, $c_1=f$,
or symmetrically, 
since all the other nonzero pairs are $\delta$-related.
By symmetry, we may  assume that $c_0=c$ and $c_1=d$.
Relabel $b$ and $a$ from \eqref{d1} as  
$w +b$ and $z+a$, where  
$ w, z \in \{ 0, x, y, x+y\}$
and $b, a \in F$
and suppose
$w +b \leq z+a +c$
and
$w +b \leq z+a +d$.
By \eqref{sum} 
$b \leq a +c$
and
$b \leq a +d$.
Notice that $\mathbf F$ is ordered like a distributive lattice, 
since it is order-isomorphic (and join semilattice isomorphic)
to the lattice on $\mathcal P( \{ c,d,e,f \} )$
with $\cup$ and $\cap$. 
From
$b \leq a + c$ and $b \leq a +d$
we get $b \leq (a + c)(a+d)= a+cd=a$.
Again by \eqref{sum},
$w \setminus z \leq a +c $
and   $w \setminus z \leq a +d $.
Suppose, say, $ w \setminus z = x$.
We have 
$x \leq a+c$ exactly in case either
$a \geq d$ or  $a \geq e$.    
We have 
$x \leq a+d$ exactly in case either
$a \geq c$ or   $a \geq f$.    
Thus if both $x \leq a+c$ and $x \leq a+d$ hold, then either
$a \geq c+d$, $a \geq c+e$, $a \geq f+d$ or $a \geq f+e$.
In each case $a \geq x$, by \eqref{fd2}.   
By \eqref{sum}, and since we have proved that 
$b \leq a$, we get  $w +b \leq z+a $.
The case  $ w \setminus z = y$ is symmetrical.
If  $ w \setminus z = x+y$, then
$x+y \leq a+c$ exactly in case either
$a \geq d$ or $a \geq e+f$.
Moreover,  
$x+y \leq a+d$ exactly in case either
$a \geq c$ or $a \geq e+f$.
Thus if both $x+y \leq a+c$ and $x+y \leq a+d$ hold, then either
$a \geq c+d$ or $a \geq e+f$.
In both cases $a \geq x+y$,
hence 
$w +b \leq z+a $, again  by \eqref{sum}.
  
\end{examples}

\begin{problems} \labbel{prob}
(a) Study posets and semilattices with a weak pre-contact relation,
namely, a binary relation satisfying only \eqref{emp},
 \eqref{ext}, \eqref{ref}, not necessarily \eqref{sym}.
See, e.~g., \cite{DV} for the additive case in Boolean algebras.

The problem is also connected with the proposal from \cite{mtt}
briefly hinted in the introduction. As well-known, if, among subsets
of a topological space, we define a relation $\delta$ by
$x \ddd y$ if $x \cap Ky \neq \emptyset  $, then $\delta$ is a weak
pre-contact, which is preserved under images of continuous
functions. 

\arxiv{(b)  Is the class of weak contact semilattices described by
Theorem \ref{thm}  finitely axiomatizable?
This is a problem asked in \cite[Section 8]{I}.
We expect that the answer is negative, but we have not worked out
counterexamples. Possibly, variations on Example \ref{ex}(d)
might work.  On the other hand, the class from Theorem  \ref{thmb} 
is indeed finitely axiomatizable, as shown by clause (3). }

(c) Characterize the classes of weak contact semilattices which can be
semilattice embedded into weak contact modular lattices
(with additive contact, with overlap contact).
Compare Example \ref{ex} (b). 
\end{problems}

\end{document}